\documentclass[12pt,times]{article}

\usepackage{graphicx}
\usepackage{epsfig}
\usepackage{amssymb}
\usepackage[left,pagewise,mathlines]{lineno}

\newtheorem{theorem}{Theorem}

\newtheorem{lemma}[theorem]{Lemma}
\newtheorem{corollary}[theorem]{Corollary}
\newtheorem{claim}[theorem]{Claim}

\newenvironment{proof}{
\par
\noindent {\bf Proof.}\rm}%
{\mbox{}\hfill\rule{0.5em}{0.809em}\par}

\title{\bf \normalsize \Large Some Results on the Target Set Selection Problem}

\author{\small Chun-Ying~Chiang\thanks{Partially supported by
 National Science Council under grant NSC97-2628-M-008-018-MY3.},
 Liang-Hao~Huang\thanks{Partially supported by
 National Science Council under grant NSC98-2811-M-008-072.},
 Bo-Jr~Li, Jiaojiao~Wu,
 Hong-Gwa
Yeh\thanks{Partially supported by National Science Council under
grant NSC97-2628-M-008-018-MY3}
 \thanks{Corresponding author (hgyeh@math.ncu.edu.tw)}\\
{\footnotesize \em Department of Mathematics, National Central
University, Taiwan}\\
{\footnotesize \em Department of Applied Mathematics, National Sun
Yat-sen University, Taiwan}}

\date{}

\begin{document}

\maketitle
\baselineskip=17pt

\begin{abstract}
 In this paper we consider a
 fundamental problem in the area of viral marketing, called
 T{\scriptsize ARGET} S{\scriptsize ET} S{\scriptsize ELECTION}
    problem.
 We study the
    problem when the underlying graph is a
   block-cactus graph, a chordal graph or a Hamming graph.
   We show that
   if $G$ is a block-cactus graph, then
   the T{\scriptsize ARGET} S{\scriptsize ET} S{\scriptsize ELECTION}
   problem can be solved in linear time,
   which generalizes Chen's result \cite{chen2009} for trees,
   and the time complexity is much better than
   the algorithm in
   \cite{treewidth} (for bounded treewidth graphs)
   when restricted to block-cactus graphs.
   We show that
   if the underlying graph $G$ is a chordal graph with thresholds
   $\theta(v)\leq 2$ for each vertex $v$ in $G$,
   then the
    problem  can be solved in linear time.
        For a Hamming graph $G$
    having thresholds $\theta(v)=2$ for each vertex $v$ of $G$,
    we precisely determine an optimal target set $S$ for
    $(G,\theta)$.
    These results partially answer  an open problem raised by
    Dreyer and Roberts \cite{Dreyer2009}.
\medskip

\noindent
 {\bf \em Key words:} target set selection, viral marketing,
 tree, block graph, block-cactus graph, chordal graph,
 Hamming graph,
  social networks, diffusion of innovations.
\end{abstract}

  %
  %
  \section{Introduction and preliminaries}
  \label{intro}
  A {\em graph} $G$ consists of a set $V(G)$ of {\em vertices} together
  with a set $E(G)$ of unordered pairs of vertices called {\em edges}.
  We use $uv$ for an edge $\{u,v\}$. Two vertices $u$ and $v$ are
  {\em adjacent} to each other if $uv \in E(G)$. In this paper, all
  graphs are finite and have no loops or multiple edges.
   For $S\subseteq V(G)$, the {\em
  subgraph of $G$ induced by $S$} is the graph $G[S]$ with vertex set
  $S$ and edge set $\{ uv \in E(G): u,v \in S\}$. Denote by $G-S$ the
   subgraph of $G$ induced by $V(G)\setminus S$ and, for convenience,
   we write $G-v$ for $G-\{v\}$ when $v\in V(G)$.
   The {\em neighborhood} of a vertex $v$ in $G$ is the
   set $N_G(v) = \{u \in V(G): uv \in E\}$.
   The {\em degree} $d_G(v)$ of  $v$  is defined by
   $d_G(v)=|N_G(v)|$.
   The {\em distance} $d_G(x,y)$ of two vertices $x$ and $y$ in $G$ is
   defined to be the length of the shortest path from $x$ to $y$ in
   $G$.
   A {\em complete graph} is a graph in which
   every two distinct vertices are adjacent.
   The complete graph on $n$ vertices is denoted by $K_n$.
   The {\em $n$-cycle} is the graph $C_n$
   with $V(C_n)=\{v_1, v_2, \ldots, v_n\}$ and $E(C_n)=\{v_1v_2, v_2v_3,
   \ldots, v_{n-1}v_n, v_nv_1\}$.
   The {\em $n$-path} is the graph $P_n$ with
   $V(P_n)=\{v_1, v_2, \ldots, v_n\}$ and $E(P_n)=\{v_1v_2, v_2v_3,
   \ldots, v_{n-1}v_n\}$.

  %
  %
   The topology of a person-to-person recommendation social network
   is usually modeled by a graph $G$ in which
   the vertices $V(G)$ represent customers, and edges $E(G)$ connect people to
   their friends.
   %
   %
   %
   %
 Consider the following scenario: A
 company wish to market a new product.
 The company has at hand a description of the social network $G$
 formed among a sample of potential customers.
 The company wants to target key potential customers $S\subseteq V(G)$ of the social
 network and persuade them into adopting the new product
 by handing out free samples.
 We assume that individuals in $S$ will be convinced to adopt the new
 product after they receive a free sample, and
 the friends of customers in $S$ would be
 persuaded into buying the new product, which in turn will
 recommend the product to other friends.
 The company hopes that by word-of-mouth effects,
 convinced vertices in $S$
 can trigger a cascade of further adoptions, and
 many customers will ultimately be persuaded.
  %
  %
 This advertising technique of spreading commercial message via social networks $G$
 is called {\em viral marketing} by analogy with
 computer viruses.
 But now how to find a good set of potential customers $S$
 to target?

  %
  %
  %
  In general, each vertex $v$ is assigned a threshold value
  $\theta(v)$. The thresholds represent the different latent tendencies of
  vertices (customers) to buy the new product when their neighbors
  (friends) do.  To be precise,
   %
   %
   %
   %
   let $G$ be a connected undirected graph equipped
   with {\em thresholds} $\theta:V(G)\rightarrow \mathbb{Z}$.
   Denote by $(G,\theta)$ the social network $G$ equipped with thresholds $\theta$.
   %
   %
   %
   %
   When $\theta$ is a constant function such that $\theta(v)=k$ for
   all vertices $v$, $(G,\theta)$ will be written as $(G,k)$ for
   short.
   %
   %
   %
   %
   Vertices $v$ of $G$ are in one of two states,
   active or inactive,
   which indicate whether $v$ is persuaded into buying the
   new product.
   We call a vertex $v$ {\em active}
   if it has been convinced to adopt the new product
   and assume that vertex $v$ becomes active
   if $\theta(v)$ of its neighbors have adopted the new product.

  %
  %
  %
  In this paper we consider the following repetitive process,
  called {\em activation process} in
  $(G,\theta)$ starting at {\em target set} $S\subseteq V(G)$,
  which unfolds in discrete steps.
  Initially (at time $0$), set all vertices in $S$
  to be active (with all other vertices inactive).
  After that, at each time step, the states of vertices
  are updated according to following rule:
  %
  %
  %
  %
   \begin{description}
          \item[Parallel updating rule:]
          All inactive vertices $v$ that have at least
          $\theta(v)$ already-active neighbors become active.
   \end{description}

  The activation
  process terminates when no more vertices can get activated.
  %
  %
  %
  %
  %
  Let
  $[S]^G_\theta$
  denote the set of vertices that are active at the end
  of the process.
  %
  %
  %
  %
  If $F\subseteq [S]^G_\theta$, then we say that the target set $S$ {\em influences}
  $F$ in $(G,\theta)$.
  %
  %
  %
  %
  Notice that if $v$ has threshold
  $\theta(v)> d_G(v)$ and $v\in [S]_\theta^G$ for some target set $S$, then
  it must be $v\in S$. We also note that, according to our rule, if
  an inactive vertex $v$ has threshold
  $\theta(v)\leq 0$ at time step $t$,
  then it becomes active automatically at the next time step.
  We are interested in the following optimization problem:
     \begin{description}
          \item[T{\scriptsize ARGET} S{\scriptsize ET}
          S{\scriptsize ELECTION}:]
          Finding a target set $S$ of smallest possible size
  that influences all vertices in the social network $(G,\theta)$,
  that is $[S]^G_\theta=V(G)$.
   \end{description}
   We define {\rm min-seed}$(G,\theta)$ to be the minimum size of a target set
  that guarantees that all 
  vertices in $(G,\theta)$ are eventually active at the end
  of the activation process, that is,
  $\mbox{min-seed}(G,\theta)=\min\{|S|: S\subseteq V(G)$ and $[S]^G_\theta=V(G)\}$.
  %
  %
  For
  $S\subseteq V(G)$,  if $[S]^G_\theta=V(G)$
  and $|S|=\mbox{\rm min-seed}(G,\theta)$, then we call $S$
  an {\em optimal target set}  for $(G,\theta)$.

  %
  %
  %
  %
  %
  Domingos and Richardson \cite{domingos2001} considered
  T{\scriptsize ARGET} S{\scriptsize ET} S{\scriptsize ELECTION}
  problem in a probabilistic setting and presented heuristic
  solutions.
  Kempe, Kleinberg, and Tardos \cite{kempe2003}
  considered probabilistic thresholds, called linear threshold model,
  and focused on the maximization version of the
  T{\scriptsize ARGET} S{\scriptsize ET} S{\scriptsize ELECTION}
  problem $-$ for any given $k$, find a target set $S$ of size $k$
  to maximize the expected number of active vertices at the end of
  the activation process. They showed that this problem is NP-hard
  and proved that a hill-climbing algorithm can efficiently
  obtain an approximation solution that is $63\%$ of optimal.

  In this paper we only consider
  the
  T{\scriptsize ARGET} S{\scriptsize ET} S{\scriptsize ELECTION}
  problem with deterministic, explicitly given, thresholds.
  In 2002, Peleg \cite{Peleg2002} showed that
  this problem is NP-hard for majority thresholds,
  that is $\theta(v)=\lceil d_G(v)/2\rceil$ for each vertex $v$ in
  $G$. In 2009, Dreyer and Roberts \cite{Dreyer2009} showed that
  the problem is NP-hard for constant thresholds $-$ given a fixed
  $k\geq 3$, $\theta(v)=k$ for each vertex $v$ in $G$, and
  Chen \cite{chen2009} proved that it is NP-hard
  for bounded bipartite graphs $G$ with thresholds at most 2.

  In general, the
  T{\scriptsize ARGET} S{\scriptsize ET} S{\scriptsize ELECTION}
  problem is not just NP-hard but also extremely hard to approximate.
  Kempe, Kleinberg, and Tardos \cite{kempe2003} showed that
  a maximization version of
  T{\scriptsize ARGET} S{\scriptsize ET} S{\scriptsize ELECTION}
  with constant thresholds
  cannot be approximated within any non-trivial factor,
  unless ${\rm P}={\rm NP}$.
  In 2009, Chen \cite{chen2009} proved that
  given any $n$-vertices regular graph with
   thresholds
   $\theta(v)\leq 2$
   for any vertex $v$,
   the T{\scriptsize ARGET} S{\scriptsize ET} S{\scriptsize ELECTION}
    problem can
   not be approximated within the ratio of
   $O(2^{\log^{1-\epsilon}n})$,
    for any fixed constant $\epsilon >0$,
   unless ${\rm NP}\subseteq {\rm DTIME}(n^{{\sf poly}\log(n)})$.

   Very little is known about min-seed$(G,\theta)$ for
   specific classes of graphs $G$.
   Dreyer and Roberts \cite{Dreyer2009} showed that when
   $G$ is a tree,
   the T{\scriptsize ARGET} S{\scriptsize ET} S{\scriptsize ELECTION}
    problem can be solved in linear time for constant thresholds.
   Chen \cite{chen2009} showed that when
   the underlying graph is a tree, the problem can be solved in polynomial-time
   under a general threshold model.
   In 2010,
   Ben-Zwi, Hermelin, Lokshtanov and Newman \cite{treewidth} showed that
   for  $n$-vertices graph $G$ with treewidth bounded by $\omega$,
    the T{\scriptsize ARGET} S{\scriptsize ET} S{\scriptsize ELECTION}
    problem can be solved  in $n^{O(\omega)}$ time.
    In \cite{Dreyer2009,tori2004}, min-seed$(G,\theta)$ is computed
    for paths, cycles and for different kinds of grids $G$
    under constant threshold model.

  %
  %
   The objective of this paper is to study the
   T{\scriptsize ARGET} S{\scriptsize ET} S{\scriptsize ELECTION}
    problem when the underlying graph is a
   block-cactus graph, a chordal graph or a Hamming graph.
   In Section \ref{sect-lock-cactus}, we show that
   if $G$ is a block-cactus graph, then
   the problem can be solved in linear time,
   which generalizes Chen's result \cite{chen2009} for trees,
   and the time complexity is much better than
   the algorithm in
   \cite{treewidth} (for bounded treewidth graphs)
   when restricted to block-cactus graphs.
   In Section \ref{sect-chordal}, we show that
   if the underlying graph $G$ is a chordal graph with thresholds
   $\theta(v)\leq 2$ for each vertex $v$ in $G$,
   then the
   T{\scriptsize ARGET} S{\scriptsize ET} S{\scriptsize ELECTION}
    problem  can be solved in linear time.
    Our results partially answer  an open problem raised by
    Dreyer and Roberts at the end of their paper \cite{Dreyer2009}.
    In Section \ref{sect-hamming-graphs}, for a Hamming graph $G$
    having thresholds $\theta(v)=2$ for each vertex $v$ of $G$,
    we precisely determine an optimal target set $S$ for
    $(G,\theta)$.

  %
  %
  %
  %
   In order to study $\mbox{min-seed}(G,\theta)$
   we introduce a sequential version of the above activation process,
   called {\em sequential activation process},
   which employs the following rule instead of the parallel updating rule:
    \begin{description}
          \item[Sequential updating rule:]
          At each time step $t$,
          exactly one of
          inactive vertices that have at least $\theta(v)$
          already-active neighbors becomes active.
   \end{description}
   %
   %
   %
   %
   %
   %
  %
  %
   The proof of the following lemma is straightforward and
   so is omitted. In the sequel, Lemma \ref{seq=para}
   will be used without explicit reference to it.
   \begin{lemma}\label{seq=para}
   For a social network $(G,\theta)$,
   an optimal target set under sequential updating
   rule is also an  optimal target set under parallel updating
   rule, and vice versa.
   \end{lemma}
  %
  %
   Let ${\cal P}$ be a sequential activation process on $(G,\theta)$
   starting out from a target set $S$.
   In this process, if $v_1,v_2,\ldots,v_r$ is the order that vertices
   in $[S]^G_\theta\setminus S$ are convinced,
   then $(v_1,v_2,\ldots,v_r)$ is called the {\em convinced sequence} of ${\cal P}$,
   and we say that target set $S$ has a convinced sequence $(v_1,v_2,\ldots,v_r)$
   on $(G,\theta)$.

   %
   %
   %
   %

   \section{Block-cactus graphs}
   \label{sect-lock-cactus}
  %
  %
  A vertex $v$ of a graph is called a {\em cut-vertex}
  if removal of $v$ and all edges incident to it
  increases the number of connected components.
  A {\em block} of a graph $G$ is a maximal connected
  induced subgraph of $G$ that has no cut-vertices.
  A graph $G$ is a {\em block graph}
  if every block of $G$ is a complete graph.
  A block $B$ of a graph $G$ is called
  a {\em pendent block} of $G$ if $B$ has at most one cut-vertex of $G$.
  A graph $G$ is a {\em block-cactus graph} if every block of $G$
  is either a complete graph or a cycle.
  %
  %
  %
  %
  Let $v$ be a cut-vertex of $G$. If $G-v$ consists of two disjoint
  graphs $W_1$ and $W_2$ and let $G_i$ $(i=1,2)$ be the subgraph of
  $G$ induced by $\{v\}\bigcup V(W_i)$, then $G$ is called the {\em
  vertex-sum} at $v$ of the two graphs $G_1$ and $G_2$, and denoted
  by $G=G_1 {\oplus_v} G_2$.

  %
  %
  %
  %
  %
  In the following Theorem \ref{cut-vertex}, let $G_1\oplus_v G_2$ be a social network
  equipped with threshold function $\theta$.
  %
  %
  Let $\theta_1$
  be a threshold function of $G_1-v$
  which is the same as the function $\theta$,
  except that $\theta_1(x)=\theta(x)-1$ for every $x\in N_{G_1}(v)$.
  %
  %
  Let $S_1$ be an optimal target set for $(G_1-v,\theta_1)$
  that maximizes the cardinality of the set
  $N_{G_1}(v) \cap [S_1]^{G_1}_\theta$,
  where, by slight abuse of notation,
  $\theta$ also means the threshold function of $G_1$ by
  restricting the threshold $\theta$ of $G_1\oplus_v G_2$
  to the set $V(G_1)$.
  %
  %
  Let $\theta_2$ be a threshold function of $G_2$ which is the same
  as the function $\theta$, except that
  $\theta_2(v)=\theta(v)-|N_{G_1}(v) \cap [S_1]^{G_1}_\theta|$.
  %
  %
  Let $S_2$ be an optimal target set for $(G_2,\theta_2)$.
  Now, with the definitions and notation introduced in this paragraph,
  we prove the following theorem.

  %
  %
  %
  %
  %
  \begin{theorem} \label{cut-vertex}
  $S_1\cup S_2$ is an optimal target set for
   $(G_1\oplus_v G_2,\theta)$.
  \end{theorem}
  %
  %
  \begin{proof}
  Consider a sequential activation process in $(G_1\oplus_v G_2,\theta)$ starting at
  target set $S_1\cup S_2$.
  Clearly $N_{G_1}(v) \cap[S_1]^{G_1}_\theta\subseteq [S_1]^{G_1\oplus_v G_2}_\theta$,
  and hence $V(G_2)\subseteq [S_1\cup S_2]^{G_1\oplus_v G_2}_\theta$,
  which implies $V(G_1)\subseteq [S_1\cup S_2]^{G_1\oplus_v
  G_2}_\theta$.
  That is the target set $S_1\cup S_2$
  influences all vertices in $(G_1\oplus_v G_2,\theta)$.
  To prove the theorem it remains to show that $|S_1|+|S_2|=$
  min-seed$(G_1\oplus_v G_2,\theta)$.

  %
  %
  Let $S$ be an optimal target set for
  $(G_1\oplus_v G_2,\theta)$ that minimizes the size of the set $S\cap V(G_1-v)$.
  Since $(S\cap V(G_1-v))\cup \{v\}$ influences all vertices in
  $(G_1,\theta)$, we have that
  $S\cap V(G_1-v)$ influences all vertices in $(G_1-v,\theta_1)$.
  It now follows that $|S\cap V(G_1-v)|=|S_1|$ since if not,
  then we have $|S\cap V(G_1 -v)|\geq |S_1|+1$, and hence
  $(S_1+v)\cup (S\cap V(G_2))$ is an optimal target set for
  $(G_1\oplus_v G_2,\theta)$, a contradiction to the choice of $S$.
  %
  %
  %
  %
  Therefore $S\cap V(G_1-v)$ is an optimal target set for
  $(G_1-v,\theta_1)$. By the choice of $S_1$, we see that
  $|N_{G_1}(v) \cap[S\cap V(G_1-v)]^{G_1}_\theta|\leq |N_{G_1}(v)
  \cap[S_1]^{G_1}_\theta|$.
  This implies that $S_1\cup [S\cap V(G_2)]$ is an optimal target set
  for $(G_1\oplus_v G_2,\theta)$, and hence
  $S\cap V(G_2)$ influences all vertices in $(G_2,\theta_2)$,
  which implies $|S\cap V(G_2)|\geq |S_2|$.
  %
  %
  %
  %
  We conclude that $|S_1|+|S_2|=|S\cap V(G_1-v)|+|S_2|\leq
  |S\cap V(G_1-v)|+|S\cap V(G_2)|=|S|$.
  Therefore $S_1\cup S_2$ is an optimal target set for
   $(G_1\oplus_v G_2,\theta)$.
  \end{proof}
  %
  %
  %
  %
  %
  \begin{corollary}
  {\rm min-seed}$(G_1\oplus_v G_2,\theta)=$
  {\rm min-seed}$(G_1-v,\theta_1)+$
  {\rm min-seed}$(G_2,\theta_2)$.
  \end{corollary}

  %
  %
  %
  %
  %
  \begin{lemma} \label{K_n-C_n}
  Let $v$ be a vertex in the social network $(G,\theta)$.
  If $G\in \{K_n, C_n\}$, then an optimal target set $S$
  for $(G-v,\theta_1)$
  that maximizes the size of the set $N_{G}(v) \cap [S]^{G}_\theta$
  can be found in linear time, where $\theta_1$
  is the threshold function of $G-v$
  which is the same as the function $\theta$,
  except that $\theta_1(x)=\theta(x)-1$ for every $x\in N_{G}(v)$.
  Moreover the size of the set $N_{G}(v) \cap [S]^{G}_\theta$ can
  also be determined in linear time.
  \end{lemma}
  %
  %
  \begin{proof}
  Let ${\cal F}$ be the set of
  optimal target sets $S$ for
  $(G-v,\theta_1)$ such that $S$ maximizes the size of the set $N_G(v)\cap
 [S]_\theta ^{G}$.

  We first consider the case that $G=K_n$.
  Let $V(G-v)=\{v_1,v_2,\ldots,v_{n-1}\}$
  such that
  $\theta_1(v_1)\leq \theta_1(v_2)\leq \cdots \leq \theta_1(v_{n-1})$.
  Let $S$ be an optimal target set for $(G-v,\theta_1)$.
  Since any two vertices $v_i,v_{i+1}$ in $G-v$ have
  $N_{G-v}(v_i)=N_{G-v}(v_{i+1})$ and
  $\theta_1(v_1)\leq \cdots \leq \theta_1(v_{n-1})$,
   we give the following simple observation without proof.

   \smallskip
   \noindent
  {\bf Observation} {\em
  If $v_i\in S$ and $v_{i+1}\not\in S$,
  then $(S\setminus  \{v_i\})\cup \{v_{i+1}\}$
  is an optimal target set for $(G-v,\theta_1)$ and
  $|[(S\setminus  \{v_i\})\cup \{v_{i+1}\}]_\theta^G|\geq |[S]_\theta^G|$.
   }

  \smallskip
  Since $G$ is a complete graph, the above observation
  says that if min-seed$(G-v,\theta_1)=s$, then the target set
  $\{v_{n-1},v_{n-2},\ldots,v_{n-s}\}\in {\cal F}$.
  Moreover, such a target set has a convinced sequence
  $(v_1,v_2,\ldots,v_{n-s-1})$ on
  $(G-v,\theta_1)$. Now we are in a position to
  show that {\bf \sf Algorithm K} outputs
  an optimal target set $S$ for $(G-v,\theta_1)$ such that $S\in
  {\cal F}$.

  In steps 2-3 of the algorithm we see that
  min-seed$(G-v,\theta_1)\geq
  |\{v_i: \theta_1(v_i)>n-2$ and $1\leq i\leq n-1\}|=\ell$.
  In steps 4-8, we want to find the value $s$ such that
  $\{v_{n-1},v_{n-2},\ldots,v_{n-\ell}\}\cup
  \{v_{n-\ell-1},v_{n-\ell-2},\ldots,v_{n-s}\}\in {\cal F}$.
  During the $i$th iteration of the for loop in step 4, we have
  $\{v_1,v_2,\ldots,v_{i-1}\}\subseteq
  [\{v_{n-1},v_{n-2},\ldots,v_{n-s}\}]_{\theta_1}^{G-v}$.
  In step 6, when $\theta_1(v_i)>s+i-1$, in order to influence vertex
  $v_i$ in $(G-v,\theta_1)$ we need to add another
  $\theta_1(v_i)-(s+i-1)$ vertices to the set
  $\{v_{n-1},v_{n-2},\ldots,v_{n-s}\}$.
  Note that in step 5 we have $\theta_1(v_i)\leq n-2$.
  If follows that after step 5 and before step 6
  we have $n-(s+[\theta_1(v_i)-(s+i-1)])>i$.
  Therefore in step 7 if $n-s=i+1$,
  then it must be min-seed$(G-v,\theta_1)=s$,
  and hence $\{v_{n-1},v_{n-2},\ldots,v_{n-s}\}\in {\cal F}$.
  Clearly, the time complexity of {\bf \sf Algorithm K} takes linear
  time, where the bucket sort algorithm is used to sort vertices by
  their thresholds.

  Let $S$ be the output of the {\bf \sf Algorithm K} and $|S|=s$.
  Let
  $V(G)\setminus S=\{u_1,u_2,\ldots,u_{n-s}\}$ such that
  $\theta(u_1)\leq \theta(u_2)\leq \ldots \leq \theta(u_{n-s})$.
  Let $U=\{i: \theta(u_i)> s+i-1$ and $1\leq i\leq n-s\}$.
  We define the value $r$ by
   \begin{linenomath}
  $$
   r=\left\{%
  \begin{array}{ll}
    \min U-1, & \hbox{if $U\not= \emptyset$;} \\
    n-s, & \hbox{if $U= \emptyset$.} \\
  \end{array}%
  \right.
  $$
   \end{linenomath}
  Since $G$ is a complete graph, it can be seen that
  $[S]_\theta^G=S\cup \{u_1,u_2,\ldots,u_r\}$.
  Therefore the size of the set $N_{G}(v) \cap [S]^{G}_\theta$ can
  also be determined in linear time.

 \begin{linenomath}
 \begin{tabbing}
      xxx\=xx\=xx\=xx\= \kill
      {\bf Algorithm K} \> \> \> \\
      {\bf Begin} \> \> \> \\
     1 \> $s\leftarrow 0$; \> \> \\
     2 \> {\bf for} $i=1$ {\bf to} $n-1$
          {\bf do} {\bf if} $\theta_1(v_i)>n-2$ {\bf then}
          $s\leftarrow s+1$;\> \> \\
     3 \> $\ell \leftarrow s$; \> \>\\
     4 \>  {\bf for} $i=1$ {\bf to} $n-\ell-1$ {\bf do} \>  \> \\
     5 \> {\bf begin}  \> \> \\
     6 \> \> {\bf if} $\theta_1(v_i)>s+i-1$
             {\bf then} $s\leftarrow s+ [\theta_1(v_i)-(s+i-1)]$;\>  \\
     7 \> \> {\bf if} $n-s=i+1$
             {\bf then} {\bf STOP} and
             {\bf output} $S=\{v_{n-1},v_{n-2},\ldots,v_{n-s}\}$;\>  \\
     8 \> {\bf end} \> \>\\
     {\bf End.} \> \> \>
  \end{tabbing}
 \end{linenomath}

  Finally, consider the remaining case that $G=C_n$.
  Let $E(G)=\{vv_1, vv_{n-1}\}\cup \{v_iv_{i+1}: 1\leq i\leq n-2\}$.
  Thus $V(G-v)=\{v_1,v_2,\ldots,v_{n-1}\}$.
  Let ${\cal H}$ be the set of
  optimal target sets $S$ for
  $(G-v,\theta_1)$.
  First we consider the following {\sf Algorithm C} which computes
  $S_1$ and $S_2$.
  Clearly, $S_1\subseteq S$ for each $S\in {\cal H}$
  and $S_2\subseteq [S_1]^{G-v}_{\theta_1}$.

  \begin{linenomath}
  \begin{tabbing}
      xxx\=xx\=xx\=xx\= \kill
      {\bf Algorithm C} \> \> \> \\
      {\bf Begin} \> \> \> \\
     1 \> Find the set $S_1=\{v_i: \theta_1(v_i)> d_{G-v}(v_i)$
                                    and $ 1\leq i\leq n-1\}$. \> \> \\
     2 \> {\bf for} $i=1$ {\bf to} $n-1$
          {\bf do} {\bf if} $v_i\not\in S_1$ {\bf then}
          $\theta_1(v_i)\leftarrow \theta_1(v_i)-|N_{G-v}(v_i)\cap S_1|$;
          \> \> \\
     3 \> {\bf for} $i=1$ {\bf to} $n-2$
          {\bf do} {\bf if} $v_i\not\in S_1$ and $\theta_1(v_i)\leq 0$ {\bf then}
          $\theta_1(v_{i+1})\leftarrow \theta_1(v_{i+1})-1$;\> \> \\
     4 \> {\bf for} $i=n-1$ {\bf downto} $2$
          {\bf do} {\bf if} $v_i\not\in S_1$ and $\theta_1(v_i)\leq 0$ {\bf then}
          $\theta_1(v_{i-1})\leftarrow \theta_1(v_{i-1})-1$;\> \> \\
     5 \> Find the set $S_2=\{v_i\not\in S_1: \theta_1(v_i)\leq 0$
                                       and $ 1\leq i\leq n-1\}$. \> \> \\
     6 \> {\bf output} $S_1$ and $S_2$; \> \>\\
     7 \> {\bf output} $\theta_1$; \> \>\\
     {\bf End.} \> \> \>
  \end{tabbing}
  \end{linenomath}

  In the sequel, let $S_1,S_2,\theta_1$ be the outputs of the {\sf Algorithm
  C}.
  Now let $G-v-S_1-S_2$ have exactly $r$ connected components
  $P_1,P_2,\ldots,P_r$.
  Denote by $\ell_i$ the value $\min\{k: v_k\in V(P_i), 1\leq k\leq n-1\}$.
  We assume that $\ell_1< \ell_2 <\cdots <\ell_r$.
  For each $1\leq i\leq r$, we note that $P_i$ is a path
  and all vertices $w$ in $P_i$ have $\theta_1(w)\in \{1,2\}$,
  moreover the two end-vertices $w_1,w_2$ of $P_i$ have $\theta_1(w_1)=\theta_1(w_2)=1$.
  Let $V(P_1)=\{v_a,v_{a+1},\ldots,v_{a+b}\}$
  and $V(P_r)=\{v_c,v_{c+1},\ldots,v_{c+d}\}$
  for some integers $a,b,c,d$.

  {\bf Case 1.} $r=1$.
  Let $\{u\in V(P_1): \theta_1(u)=2\}=
                         \{v_{i_1},v_{i_2},\ldots,v_{i_q}\}$ such that
                         $i_1<i_2<\cdots <i_q$.
  If $q=0$, then $S_1\cup\{v_a\},S_1\cup\{v_{a+b}\}\in {\cal H}$.
  Clearly either $S_1\cup\{v_a\}\in {\cal F}$
  or $S_1\cup\{v_{a+b}\}\in {\cal F}$.
  It follows that we can compute $[S_1\cup\{v_a\}]_\theta^G$ and
  $[S_1\cup\{v_{a+b}\}]_\theta^G$ to find a desired set $S$ in ${\cal F}$.
  When $q=2t$ for some $t\in \mathbb{Z}^+$,
  let $U_1= \{v_{a}\}\cup \{v_{i_2}, v_{i_4},\ldots,v_{i_{2t}}\} $
  and
  $U_2= \{v_{i_1}, v_{i_3},\ldots,v_{i_{2t-1}}\}\cup \{v_{a+b}\}$.
  It can be seen that either $S_1\cup U_1\in {\cal F}$ or
  $S_1\cup U_2\in {\cal F}$.
  One can compute
  $[S_1\cup U_1]_\theta^G$ and
  $[S_1\cup U_2]_\theta^G$ to find a desired set $S$ in ${\cal F}$.
  When $q=2t-1$ for some $t\in \mathbb{Z}^+$,
  let
  $U= \{v_{i_1}, v_{i_3},\ldots,v_{i_{2t-1}}\}$.
  Clearly $S_1\cup U\in {\cal F}$.

 {\bf Case 2.} $r\geq 2$. It suffices to assume that $r=3$, that is
 $G-v-S_1-S_2$ has exactly $3$ connected components
  $P_1,P_2,P_3$ and $\ell_1< \ell_2 <\ell_3$.
  Let $\{u\in V(P_1): \theta_1(u)=2\}=
                         \{v_{i_1},v_{i_2},\ldots,v_{i_q}\}$ such that
                         $i_1<i_2<\cdots <i_q$.
  Let $\{u\in V(P_2): \theta_1(u)=2\}=
                         \{v_{j_1},v_{j_2},\ldots,v_{j_s}\}$ such that
                         $j_1<j_2<\cdots <j_s$.
  Let $\{u\in V(P_3): \theta_1(u)=2\}=
                         \{v_{k_1},v_{k_2},\ldots,v_{k_\ell}\}$ such that
                         $k_1<k_2<\cdots <k_\ell$.
  It suffices to consider the case that $q=2t,s=2t'-1, \ell=2t''$
  for some integers $t,t',t''$
  (the remaining cases follow similar arguments as above).
  let $U_1= \{v_{a}\}\cup \{v_{i_2}, v_{i_4},\ldots,v_{i_{2t}}\} $,
  $U_2= \{v_{j_1}, v_{j_3},\ldots,v_{j_{2t'-1}}\} $,
  and
  $U_3= \{v_{k_1}, v_{k_3},\ldots,v_{k_{2t''-1}}\}\cup \{v_{c+d}\}$.
  It can be seen that $S_1\cup U_1\cup U_2\cup U_3\in {\cal F}$.

  Concerning the running time of the above algorithm,
  it is clear that it is linear time.
  Which completes the proof of the lemma.
  \end{proof}

  Now Theorem \ref{block-cactus} follows from Theorem \ref{cut-vertex} and
  Lemma \ref{K_n-C_n} immediately.
  %
  %
  %
  %
  %
  %
  \begin{theorem} \label{block-cactus}
  If $G$ is a block-cactus graph, then an optimal target set for
  $(G,\theta)$ can be found in linear time.
  \end{theorem}

  \section{Chordal graphs}
  \label{sect-chordal}
  %
  %
  %
  %
  %
  %
  A graph is called {\em chordal} if it does not have an induced
  cycle of length greater than three.
  A vertex $v$ in $G$ is called {\em simplicial} if the subgraph of
  $G$ induced by the neighbors of $v$ is complete.
  Let $\sigma=[v_1,v_2,\ldots,v_n]$ be an ordering of $V(G)$.
  We say that $\sigma$ is a {\em perfect elimination order}
  if each $v_i$ is a
  simplicial vertex of the subgraph $G[v_i,v_{i+1},\ldots,v_n]$.
  In 1965, Fulkerson and Gross \cite{fg65}
  showed that every chordal graph has
  a perfect elimination order.
  %
  %
  %
  %
  %
  %
  In \cite{vertex-elimin,MCS-Tarjan-Yann} it was shown that
  if $G$ is a chordal graph, then there is a linear time algorithm
  which receives the adjacency sets of $G$ and
  outputs a perfect elimination order $\sigma$ of $V(G)$.
  %
  %
  %
  %
  %
  For nonadjacent vertices $u$ and $v$ of a graph $G$, a subset
  $S\subseteq V(G)$ is called a {\em $u$-$v$ separator} if the
  removal of $S$ from $G$ separates $u$ and $v$ into distinct
  connected components.
  If no proper subset of $S$ is a $u$-$v$-separator, then $S$
  is called a {\em minimal $u$-$v$ separator}.

  %
  %
  %
  %
  %
  %

 %
 %
 %
 %
 \begin{lemma}[\cite{dirac61}]
 \label{Non-adjacent-simplicial}
 Every chordal graph $G$ has a simplicial vertex.
 Moreover, if $G$ is not complete, then it
 has two nonadjacent simplicial vertices.
 \end{lemma}

 %
 %
 %
 %
 %
 \begin{lemma}[\cite{fg65}]
 \label{separator}
 For nonadjacent vertices $u$ and $v$ of a chordal graph $G$,
 if $S$ is a minimal $u$-$v$ separator of $G$, then
 $S$ induces a complete subgraph of $G$.
 \end{lemma}

 %
 %
 %
 %
 %
 %
 \begin{lemma}\label{chordal-key-lemma}
 For $t\geq 2$, let $G$ be a $t$-connected chordal graph with $\theta(x)\leq t$ for all vertices $x$. If $S\subseteq V(G)$ induces a complete subgraph of size $t$ in $G$,
 then the target set $S$ influences all vertices in $(G,\theta)$.
 \end{lemma}

 \begin{proof}
 Without loss of generality, we may assume that $G$ is not complete.
 Let $|V(G)|=n$.
 To prove this theorem, we want to demonstrate a sequence of distinct
 vertices $[v_1,v_2,\ldots,v_{\ell}] $
 in $G$ such that
 $G- \{v_1,v_2,\cdots, v_\ell\}$ is a complete graph that contains all vertices
 of $S$.
 Moreover, for $1\leq i\leq \ell$, vertex $v_i$ is adjacent to at least
 $t$ vertices in  the graph $G- \{v_1,v_2,\cdots, v_i\}$.
 It is clear that if such a sequence exists, then the target set $S$
 influences all vertices in $(G,\theta)$,
 since $\theta(x)\leq t$ for all vertices $x$ in $G$.

 To construct such a sequence, by Lemma \ref{Non-adjacent-simplicial},
 we can pick a simplicial vertex $v_1$ of $G$
 such that $v_1\not\in S$.
 %
 %
 Note that $G-v_1$ is $t$-connected, since otherwise there is a set $U\subseteq V(G-v_1)$ with $|U|\leq t-1$ such that $G-v_1-U$ is disconnected.
 By Lemma \ref{Non-adjacent-simplicial} it follows that $G-U$ is disconnected,
 a contradiction to $G$ is $t$-connected.
 Next, if $G-v_1$ is not complete,
 then by Lemma \ref{Non-adjacent-simplicial} again,
 we can pick a simplicial vertex $v_2$ of $G-v_1$
 such that $v_2\not\in S$.
 It can also be seen that  $G-v_1-v_2$ is $t$-connected.
 If we continue in this way, we eventually have
 a desired sequence of distinct vertices $[v_1,v_2,\ldots,v_\ell]$ such that
 the graph $G- \{v_1,v_2,\cdots, v_i\}$ is $t$-connected for each
 $i\in \{1,2,\ldots,\ell-1\}$ and $G- \{v_1,v_2,\cdots, v_\ell \}$ is a complete
 graph that contains all vertices of $S$. Which completes the proof of the lemma.
 \end{proof}

  %
  %
  %
  %
  %
  \begin{theorem} \label{chordal}
  Suppose that $G$ is a $t$-connected chordal graph with $t\geq 2$.
  $(a)$ {\rm min-seed}$(G,t)=t$.
  $(b)$ If $\theta(x)\leq t$ for each vertex $x$ of $G$ and
  $\theta(v)< t$ for some vertex $v$.
  then {\rm min-seed}$(G,\theta)<t$.
  \end{theorem}
  \begin{proof}
  (a) By Lemma \ref{Non-adjacent-simplicial},
  the fact that $G$ is a $t$-connected chordal graph implies that
  $G$ contains a complete subgraph $H$ of $t$ vertices.
  By Lemma \ref{chordal-key-lemma}, we see that the target set $V(H)$
  influences all vertices in the social network $(G,t)$, and hence min-seed$(G,t)\leq
  t$.
  Note that an inactive vertex $v$ in $(G,t)$
  become active only if $v$ has at least $t$ already-active
  neighbors. It follows that min-seed$(G,t)\geq t$, which completes
  the proof of part (a).

  (b)
  If $v$ is adjacent to all other vertices of $G$,
  then, by Lemma \ref{Non-adjacent-simplicial},
  $G-v$ contains a complete subgraph $H$ of size $t-1$, since $G-v$
  has a simplicial vertex and $G$ is $t$-connected.
  It follows that, by Lemma \ref{chordal-key-lemma}, the target set $V(H)$
  influences all vertices in $(G,\theta)$, and hence
  min-seed$(G,\theta)< t$.
  Now consider the case that $v$
  is not adjacent to some vertex $u$ in $G$.
  Clearly there is a minimal $v$-$u$ separator $S$
  such that $v$ adjacent to all vertices
  of $S$. Note that $|S|\geq t$, since $G$ is $t$-connected.
  Let $S'\subseteq S$ with $|S'|=t-1$.
  By Lemma \ref{separator}, $S'\cup \{v\}$
  induces a complete subgraph of size $t$ in $G$.
  It follows that, by Lemma \ref{chordal-key-lemma} and the fact that
  $\theta(v)\leq t-1$, the target set $S'$ influences all vertices of
  $(G,\theta)$. We conclude that min-seed$(G,\theta)< t$.
  \end{proof}

 %
 %
 %
 %
 %
 %
 \begin{corollary}
 \label{chordal-corollary}
  Let $G$ be a $2$-connected chordal graph with thresholds
  $\theta(v)\leq 2$ for every vertex $v$ of $G$.
  Then {\rm min-seed}$(G,\theta)=2$ if and only if
  $\theta(v)=2$ for each vertex $v$ of $G$.
 \end{corollary}

 In the sequel, for convenience,
 we write ${\cal S}\propto  (G,\theta)$ to mean that
 the target set ${\cal S}$ influences all vertices in $(G,\theta)$.
  The following simple fact, which we state without proof, will be used implicitly and frequently in
 Lemma \ref{main-lemma-chordal-algorithm}.
 \begin{claim}
 \label{S-plus-v}
 Let $v$ be a vertex in the social network $(G,\theta)$ and
 let $\theta_1$ be the
 threshold function of $G-v$ which is the same as the function
 $\theta$, except that $\theta_1(x)=\theta(x)-1$ for every $x\in
 N_G(v)$.
 Then for $S\subseteq V(G-v)$, we have $S \propto (G-v,\theta_1)$
 if and only if $S\cup \{v\}  \propto (G,\theta)$.
 \end{claim}

 We state Lemma \ref{main-lemma-chordal-algorithm} using the
 same notation and conventions as in Claim \ref{S-plus-v}.
 \begin{lemma}
 \label{main-lemma-chordal-algorithm}
 Let $G$ be a $2$-connected chordal graph with thresholds
 $\theta(u)\le 2$ for every $u\in V(G)$.
 For a vertex $v$ in $G$, let ${\cal F}$ be the set of
 optimal target sets $S$ for
 $(G-v,\theta_1)$ such that $S$ maximizes the size of the set $N_G(v)\cap
 [S]_\theta ^{G}$.
  Let
  $I=\{u\in V(G-v): \theta_1(u)\leq 0\}$,
  $J=\{u\in V(G): \theta(u)<2\}$ and
 $J_0=\{u\in V(G): \theta(u)\leq 0\}$.
 Let ${\cal P}_1$ (resp. ${\cal Q}_1$) be the property that
 there are two distinct vertices $x,y\in I$ (resp. $x,y\in J_0$)
 such that   $d_{G}(x,y)\leq 2$.
 Let ${\cal P}_2$ (resp. ${\cal Q}_2$) be the property
 that there is an edge
 $xy$ in $G-v$ (resp. $G$) with $x\in I$ (resp. $x\in J_0$)
 and $\theta_1(y)=1$ (resp. $\theta(y)=1$).
 Then we have:

 $(a)$ If $I\cap N_G(v)\not = \emptyset$, then $\emptyset \in {\cal F}$.

 $(b)$ If $I\cap N_G(v) = \emptyset$ and ${\cal P}_1$ holds,
       then $\emptyset \in {\cal F}$.

 $(c)$ If $I\cap N_G(v) = \emptyset$ and ${\cal P}_2$ holds,
    then $\emptyset \in {\cal F}$.

    $(d)$ If $J=\emptyset$, then
 $\{x\}\in {\cal F}$ and  $[\{x\}]_\theta ^{G}=\{x\}$
 for every $x\in N_G(v)$.

 $(e)$ If $J\not=\emptyset$, $I\cap N_G(v) = \emptyset$
 and neither  ${\cal P}_1$ nor  ${\cal P}_2$ holds,
 then $\{x\}\in {\cal F}$ and $[\{x\}]_\theta^G=V(G)$
  for every vertex $x$ adjacent to some
 vertex $w\in J$.

 $(f)$ If ${\cal Q}_1$ or ${\cal Q}_2$ holds,
       then $[\emptyset]_\theta^G=V(G)$.

 $(g)$ If neither ${\cal Q}_1$ nor ${\cal Q}_2$ holds,
       then $[\emptyset]_\theta^G=J_0$.
 \end{lemma}

 \begin{proof}
 (a) Let $w\in I\cap N_G(v)$.
  By the facts $vw\in E(G)$,
   $\theta(w)\leq 1$ and by
   Lemma \ref{chordal-key-lemma}, we see that
    $\{v\}\propto (G,\theta)$, and hence
  $\emptyset\propto (G-v,\theta_1)$.

 (b)
  Clearly
  $\theta(x)\leq 0$ and $\theta(y)\leq 0$.
  Since $d_{G}(x,y)\leq 2$,
  either $xy\in E(G)$ or  $x,y\in N_{G}(z)$ for some vertex $z$. In both cases, by Lemma \ref{chordal-key-lemma},
  we see that $[\emptyset]_\theta^G=V(G)$.
  Thus by Claim \ref{S-plus-v},
  $\emptyset\propto (G-v,\theta_1)$.

  (c)
  Since $x\not\in N_G(v)$,
  it can be seen that $[\{v\}]_{\theta}^{G}\supseteq \{x,y,v\}$.
  By Lemma \ref{chordal-key-lemma},
   it follows that $\{v\}\propto (G,\theta)$, and hence
  $\emptyset\propto (G-v,\theta_1)$.

  (d)
 For each $x\in N_G(v)$,
 by Lemma \ref{chordal-key-lemma}, $\{x,v\}\propto (G,\theta)$, and
 hence $\{x\}\propto (G-v,\theta_1)$. Clearly min-seed$(G-v,\theta)\geq 1$. It follows that
 $\{x\}$ is an optimal target set for $(G-v,\theta_1)$.
 Since $\theta(u)=2$ for each  $u\in V(G)$, we have
 $|[S]_\theta ^{G}|=1$ for any optimal target set $S$
 for $(G-v,\theta_1)$. Therefore $\{x\}\in {\cal F}$
 and  $[\{x\}]_\theta ^{G}=\{x\}$.

  (e) Note that $I\cap N_G(v) = \emptyset$
  implies that $\theta(y)=2$ for each
  $y\in N_G(v)$. We claim that $\{v\}$ can not influence all
  vertices in $(G,\theta)$.
  If not, then it must be that
  either ${\cal P}_1$ or  ${\cal P}_2$ holds, a
  contradiction. Thus {\rm min-seed}$(G-v,\theta_1)\geq1$.
  Now let $w\in J$ and $x\in N_G(w)$. Note that $x\not=v$.
  Clearly $[\{x\}]_\theta^G\supseteq \{x,w\}$ and hence,
  by Lemma \ref{chordal-key-lemma}, $\{v,x\}\propto (G,\theta)$.
  It follows that, by Claim \ref{S-plus-v}, $\{x\}\propto (G-v,\theta_1)$.
  Moreover, we have $[\{x\}]_\theta^G=V(G)$, and hence
  $\{x\}\in {\cal F}$. This completes the proof of (e).

  Finally, by similar arguments as in the proofs of (c), (d) and (e),
  it is easy to prove (f) and (g), so we omit the proofs of (f) and
  (g).
 \end{proof}

 Using the same notation
 and conventions
 as in Claim \ref{S-plus-v}
 and Lemma \ref{main-lemma-chordal-algorithm},
 we state and prove the following theorem.

 \begin{theorem}
 If $G$ is a chordal graph with thresholds $\theta(x)\le
 2$ for each vertex $x$ in $G$, then an optimal target set for $(G,\theta)$
 can be found in linear time.
 \end{theorem}

 \begin{proof}
 Let $G_1$ be a block of $G$ which contains exactly one
 cut vertex $v$ of $G$.
 If $G$ is not $2$-connected, then $G$ can be written as the
 following form: $G=G_1\oplus_v G_2$, where $G_2$ is an
 induced subgraph of $G$ and is also chordal.
 To prove the theorem, we omit the easy case $G_1=K_2$,
 which follows from Lemma \ref{K_n-C_n}.
 We only consider the case that
 $G_1$ is a $2$-connected chordal graph.
 By using Lemma \ref{main-lemma-chordal-algorithm},
 we can in linear time in terms of the size of $G_1$
 find an optimal target sets $S_1$ for
 $(G_1-v,\theta_1)$ such that $S_1$ maximizes the size of the set
 $N_{G_1}(v)\cap [S_1]_\theta ^{G_1}$ and
 compute $|N_{G_1}(v)\cap [S_1]_\theta^{G_1}|$.

  Next, we want to find an optimal target set $S_2$ for
  $(G_2,\theta_2)$, where $\theta_2$ is a threshold function of
  $G_2$ which is the same
  as the function $\theta$, except that
  $\theta_2(v)=\theta(v)-|N_{G_1}(v) \cap [S_1]^{G_1}_\theta|$.
  If $G_2$ is a $2$-connected chordal graph, then
  $S_2$ can be found in linear time in terms of the size of $G_2$
  by using
  Lemma \ref{chordal-key-lemma} and Corollary
  \ref{chordal-corollary}, and hence
  an optimal target set $S_1\cup S_2$ for $(G,\theta)$
  can be found in linear time by using Theorem \ref{cut-vertex}.

  If $G_2$ has a cut vertex $v'$ and a pendent block $G_{21}$ such
  that $G_2=G_{21}\oplus_{v'}G_{22}$,
  then we can repeat the arguments in the previous paragraphs
  and use Theorem \ref{cut-vertex} to find the desired $S_2$ in
  linear time in terms of the size of $G_2$, and hence
 an optimal target set for $(G,\theta)$ can be found in
 linear time.
 \end{proof}

  \section{Hamming graphs}
  \label{sect-hamming-graphs}
 %
 %
 %
 %
 %
 Given two graphs $G$ and $H$, their {\em Cartesian product} is the
 graph $G\Box H$ with vertex set $V(G) \times V(H)$
 and edge set
 $\{(g,h)(g',h'):
 gg' \in E(G)$ with $h=h'$,
 or
 $g=g'$ with $hh' \in E(H)\}$.
 The Cartesian product is commutative and
 associative (see page 29 of \cite{sandi}).
 A {\em Hamming graph} is a Cartesian product of nontrivial complete graphs,
 i.e., of the form $K_{n_1}\Box K_{n_2}\Box \cdots \Box K_{n_t}$
 for some integers $n_1,\ldots,n_t\geq 2,t\geq 1$, which is also denoted as
 $\prod_{i=1}^{t}K_{n_i}$.
 Note that $\prod_{i=1}^{t}K_{n_i}$ has vertex set
 $V(K_{n_1})\times V(K_{n_2})\times \cdots \times V(K_{n_t})$.

 Let $u=(u_1,\ldots,u_{t})$ and $v=(v_1,\ldots,v_{t})$ be two
 vertices of $\prod_{i=1}^{t}K_{n_i}$.
 The {\em Hamming distance} $H(u,v)$ between $u$ and $v$ is
 the number of coordinate positions in which $u$ and $v$ differ.
 Note that there is an edge
 between $u$ and $v$
 if and only if $H(u,v)=1$.
 For $S_1,S_2\subseteq V(\prod_{i=1}^{t}K_{n_i})$, denote
 by $d(S_1,S_2)$ the value
 $\min\{H(u,v):u\in S_1, v\in S_2\}$.
 Let $[i,j]$ denote the set of integers $k$ such that
 $i\leq k\leq j$.
 For $A\subseteq [1,t]$, if $u_i=v_i$ for all $i\in A$,
 then
 we write $u_{|A}=v_{|A}$.
 Let $u_A$ denote the set of vertices $x$ in $\prod_{i=1}^{t}K_{n_i}$
 such that $x_{|A}=u_{|A}$.
 The proof of the following claim is straightforward and hence omitted.
 %
 %
 %
 \begin{claim}
 \label{claim-hamming-graph} Let $u,v,w$ be three distinct vertices of
 $\prod_{i=1}^{t}K_{n_i}$
 and $u_{|A}=v_{|A}$ for some set $A\subseteq [1,t]$.
 If $w$ is adjacent to both $u$ and $v$, then $w_{|A}=u_{|A}=v_{|A}$.
 \end{claim}

 %
 %
 %
 %
 \begin{lemma}
 \label{key_observation}
 Suppose $G=(V,E)$ is the Hamming graph $\prod_{i=1}^{t}K_{n_i}$.
 Let $x,y\in V$, $i,j\in [1,t]$ and $A,B\subseteq [1,t]$.
 The following properties hold.\\
 $(a)$
 If $xy\in E$ and $x_i\not= y_i$,
 then $[x_A\cup \{y\}]^G_2=x_{A\setminus \{i\}}$.\\
 $(b)$ $[x_A\cup x_B]^G_2=x_{A\cap B}$.\\
 $(c)$ If $xy\in E$ and $x_i\not= y_i$,
 then $[x_A\cup y_B]^G_2=x_{(A\cap B)\setminus \{i\}}$.\\
 $(d)$ If $H(x,y)=2$, $x_i\not=y_i$, $x_j\not=y_j$ and $i\not=j$,
 then $[x_A\cup y_B]^G_2=x_{(A\cap B)\setminus \{i,j\}}$.\\
 $(e)$ If $d(x_A,y_B)\geq 3$, then $[x_A\cup y_B]_2^G=x_A\cup y_B$.
 \end{lemma}
 \begin{proof} (a)
 First let us consider the case of $i\not\in A$.
 From Claim \ref{claim-hamming-graph} and the fact
 $x_{|A}=y_{|A}$, we see that $[x_A\cup \{y\}]^G_2=x_{A}$.
 Now we consider the remaining case $i\in A$.
 To prove this case it suffices to consider
 the case that $i=1$ and $A=[1,j]$.
 We want to prove, by induction on $j$, that
 $[x_{[1,j]}\cup \{y\}]^G_2=x_{[2,j]}$ for
 $j=t,t-1,\ldots,1$.
 For $j=t$, we see that
 $[x_{[1,j]}\cup \{y\}]^G_2=[\{x,y\}]^G_2$.
 Since $x_{[2,t]}=y_{[2,t]}$, it follows from Claim
 \ref{claim-hamming-graph} that if $w\in [\{x,y\}]^G_2$
 then $w_{|[2,t]}=x_{|[2,t]}=y_{|[2,t]}$, and hence $w\in x_{[2,t]}$.
 That is $[\{x,y\}]^G_2\subseteq x_{[2,t]}$.
 Since any vertex in $x_{[2,t]}\setminus \{x,y\}$ is adjacent to both $x$ and
 $y$, it follows that $[\{x,y\}]^G_2\supseteq x_{[2,t]}$.
 Therefore $[\{x,y\}]^G_2= x_{[2,t]}$.

 Next, we assume that
 $[x_{[1,j]}\cup \{y\}]^G_2=x_{[2,j]}$ holds for some
 $j\in [2,t]$.
 From this induction hypothesis it follows that
 $x_{[2,j]}\subseteq [x_{[1,j-1]}\cup \{y\}]^G_2$.
 For any vertex $w$ in $x_{[2,j-1]}$,
 either $w\in x_{[2,j]}\cup x_{[1,j-1]}$ or
 $w$ is adjacent to at least one vertex in $ x_{[2,j]}$
 and at least one vertex in $x_{[1,j-1]}$.
 Thus $x_{[2,j-1]}\subseteq [x_{[1,j-1]}\cup \{y\}]^G_2$.
 On the other hand, by the fact $x_{[2,j-1]}=y_{[2,j-1]}$ and
 Claim \ref{claim-hamming-graph}, we also see that
 $x_{[2,j-1]}\supseteq [x_{[1,j-1]}\cup \{y\}]^G_2$.
 Therefore $[x_{[1,j-1]}\cup \{y\}]^G_2=x_{[2,j-1]}$,
 this completes the proof of Lemma \ref{key_observation}(a).

 (b) Since for any $i\in A\setminus B$, there exists a vertex
 $y\in x_B$ such that $xy\in E$ and $x_i\not= y_i$,
 by Lemma \ref{key_observation}(a), it follows that
 $[x_A\cup x_B]^G_2\supseteq x_{A\setminus (A\setminus B)}
 =x_{A\cap B}$.
 We note that if a vertex $w$ is adjacent to at least two
 vertices in $x_A\cup x_B$, then, by Claim \ref{claim-hamming-graph},
  it must be the case that
 $w_{|A\cap B}=x_{|A\cap B}$. Therefore
 $[x_A\cup x_B]^G_2\subseteq x_{A\cap B}$.
 We conclude that $[x_A\cup x_B]^G_2= x_{A\cap B}$.

 (c) Lemma \ref{key_observation}(a) shows that
 $[x_A\cup y_B]^G_2\supseteq [x_A\cup \{y\}]^G_2=x_{A\setminus
 \{i\}}$, and hence $[x_A\cup y_B]^G_2=[x_{A\setminus
 \{i\}}\cup y_B]^G_2$,
 since $x_A\subseteq x_{A\setminus \{i\}}$.
  It follows that $[x_A\cup y_B]^G_2=[y_{A\setminus \{i\}}\cup y_B]^G_2
  =y_{(A\setminus \{i\})\cap B}=x_{(A\cap B)\setminus \{i\}}$,
  by Lemma \ref{key_observation}(b) and the fact that
   $x_{A\setminus \{i\}}=y_{A\setminus \{i\}}$.

   (d) Without loss of generality, consider only the
   case $\{i,j\}=\{1,2\}$.
   Clearly there exist two vertices $w,z$ in $G$
   such that $(w_1,w_2)=(y_1,x_2)$, $(z_1,z_2)=(x_1,y_2)$ and
   $w_{|[3,t]}=z_{|[3,t]}=x_{|[3,t]}=y_{|[3,t]}$.
   Since $w$ and $z$ are each adjacent to both $x$ and $y$,
   we see that $x_A\cup y_B$ influences $\{w,z\}$ in
   the social network $(G,2)$.
   It follows that $[x_A\cup y_B]^G_2\supseteq
   [x_A\cup \{w\}\cup \{z\}]_2^G$ and
   $ [x_A\cup y_B]^G_2\supseteq [y_B\cup \{w\}\cup \{z\}]^G_2$.
   By using Lemma \ref{key_observation}(a) twice, we see that
   $[x_A\cup \{w\}\cup \{z\}]_2^G\supseteq x_{A\setminus \{1,2\}}$,
   and hence
   $[y_B\cup \{w\}\cup \{z\}]_2^G\supseteq y_{B\setminus
   \{1,2\}}$.
   By the fact $y_{B\setminus  \{1,2\}}=x_{B\setminus  \{1,2\}}$ and
   using Lemma \ref{key_observation}(b), we get that
   $ [x_A\cup y_B]^G_2 \supseteq
   [x_{A\setminus  \{1,2\}}\cup x_{B\setminus  \{1,2\}}]^G_2
   =x_{(A\cap B)\setminus \{1,2\}}$.
   Since $(x_A\cup y_B)\subseteq x_{(A\cap B)\setminus \{1,2\}}$,
   we conclude that $ [x_A\cup y_B]^G_2
   =x_{(A\cap B)\setminus \{1,2\}}$.

   (e) For a vertex $w$ in $V\setminus (x_A\cup y_B)$,
   by Claim \ref{claim-hamming-graph},
   we see that $w$ cannot be adjacent to two distinct vertices in
   $x_A$ (resp. $y_B$). Note that since $d(x_A,y_B)\geq 3$
   there is no vertex $w$ in $V\setminus (x_A\cup y_B)$ that
   is adjacent to one vertex in $x_A$ and is also adjacent to one
   vertex in $y_B$. This completes the proof of (e).
 \end{proof}

 For $\mathcal{U}\subseteq [1,t]$, let $\overline{\mathcal{U}}$ denote the set
 $[1,t]\setminus \mathcal{U}$.
 Using the notation and results in Lemma \ref{key_observation},
 we immediately obtain the following:
 %
 %
 %
 %
 %
 %
 \begin{claim}
 \label{claim_set_distance}
 $(a)$ If $x_A\cap y_B\not= \emptyset$, then
 $[x_A\cup y_B]_2^G=x_{A\cap B}$.\\
 $(b)$ If $d(x_A, y_B)=1$,
 then $[x_A\cup y_B]_2^G=x_{ A\cap B\cap \overline{\{i\}}}$ for some $i\in
 [1,t]$.\\
  $(c)$ If $d(x_A, y_B)=2$,
 then $[x_A\cup y_B]_2^G=x_{A\cap B\cap \overline{\{i,j\}}}$ for some $i,j\in
 [1,t]$.
 \end{claim}

 %
 %
 %
 %
 %
 %
 \begin{theorem}
 \label{key-theorem-for-Hamming}
 Suppose $G=(V,E)$ is the Hamming graph $\prod_{i=1}^{t}K_{n_i}$
 and $S$ is a non-empty set of vertices.\\
 $(a)$ There exist vertices
 $x^1,x^2,\ldots,x^k \in V$
 and
 sets $A_1,A_2,\ldots, A_k\subseteq [1,t]$
 such that $[S]_2^G=\cup_{i=1}^k  x^i_{A_i}$
 with $d(x^i_{A_i},x^j_{A_j})\geq 3$ for any $1\leq i<j \leq k$.\\
 $(b)$ If $[S]_2^G=\cup_{i=1}^k  x^i_{A_i}$
 for some vertices
 $x^1,\ldots,x^k$ in $V$
 and some
 sets $A_1,\ldots, A_k\subseteq [1,t]$
 with $d(x^i_{A_i},x^j_{A_j})\geq 3$ for any $1\leq i<j \leq k$,
 then the following inequality holds:
 \begin{linenomath}
 $$\sum_{i=1}^k|A_i| \geq (2+t)k-2|S|.\,\,\,\,\,\,\,\, (\star)$$
 \end{linenomath}
 \end{theorem}
 \begin{proof} (a)
 Note that $S=\cup_{x\in S}x_{[1,t]}$ and
 $[S'\cup S^*]_2^G=[[S']_2^G\cup S^*]_2^G$ for any $S',S^*\subseteq V$.
   By using several times Claim \ref{claim_set_distance} and Lemma
 \ref{key_observation}(e), we can get
  vertices
  $x^1,x^2,\ldots,x^k \in V$
  and
  sets $A_1,A_2,\ldots, A_k\subseteq [1,t]$
  such that $[S]_2^G=\cup_{i=1}^k  x^i_{A_i}$
  with $d(x^i_{A_i},x^j_{A_j})\geq 3$ for any $1\leq i<j \leq k$.

 (b) To prove this part
 we use induction on the size of $S$.
 When $|S|=1$  (say $S=\{x^1\}$), since in this scenario $[S]_2^G=x^1_{[1,t]}$,
 it can be seen that the inequality $(\star)$ clearly holds.
 Now assume that the statement of  Theorem \ref{key-theorem-for-Hamming}(b) holds for any $S\subseteq V$
 having $|S|< \ell$.

 When $|S|=\ell \geq 2$, the proof is divided into
 cases according to the value of $k$.\\
 {\bf Case 1}. $k=1$. In this case, pick $x\in S$ and let $S'=S\setminus
 \{x\}$. Note that $S'$ is not empty.
 By Theorem \ref{key-theorem-for-Hamming}(a) we see that there are
 vertices
 $y^1,y^2,\ldots,y^r \in V$
 and
 sets $B_1,B_2,\ldots, B_r\subseteq [1,t]$
 such that $[S']_2^G=\cup_{i=1}^r y^i_{B_i}$
 having $d(y^i_{B_i},y^j_{B_j})\geq 3$ for any $1\leq i<j \leq r$.
 We have $x_{A_1}^1=[S]_2^G=
 [x_{[1,t]}\cup [S']_2^G]_2^G=[x_{[1,t]}\cup (\cup_{i=1}^r
 y^i_{B_i})]_2^G$.
 Then from Claim \ref{claim_set_distance} and Lemma
 \ref{key_observation}(e) we see that
  $A_1=(\cap_{i=1}^r B_{i})\cap \overline{\mathcal{U}}$ for some set
 $\mathcal{U}\subseteq [1,t]$
  having $|\mathcal{U}|\leq  2r$.
  Since $|S'|< \ell$, by the induction
  hypothesis, we have
  $\sum_{i=1}^r |B_i| \geq (2+t)r-2|S'|$. It follows that
  $|A_1|=t-|(\cup_{i=1}^r \overline{B_{i}})\cup \mathcal{U}|
  \geq t-\sum_{i=1}^r(t-|B_{i}|)-2r \geq
  t-rt+(2+t)r-2|S'|-2r=(2+t)-2|S|$. Thus inequality $(\star)$ holds
  in this case.\\
  {\bf Case 2}. $k>1$. In this case, let $S^*=S\cap x^1_{A_1}$ and
  $S'=S\setminus S^*$. Note that $S^*$ and $S'$ are not empty.
  Clearly $[S^*]_2^G=x^1_{A_1}$ and $[S']_2^G=\cup_{i=2}^k
  x^i_{A_i}$.
   By the induction hypothesis we see that $|A_1|\geq (2+t)-2|S^*|$
   and $\sum_{i=2}^k |A_i|\geq (2+t)(k-1)-2|S'|$.
   It follows immediately that inequality $(\star)$ holds in this
   case.
  This completes the proof of the theorem.
 \end{proof}


 %
 %
 %
 %
 %
 \begin{theorem}
 \label{Hamming_(G,2)}
 If $G=(V,E)$ is the Hamming graph $\prod_{i=1}^{t}K_{n_i}$,
 then {\rm min-seed}$(G,2)=1+\lceil{t \over 2}\rceil$.
 \end{theorem}

 \begin{proof}
 Note that $V=V(K_{n_1})\times V(K_{n_2})\times \cdots \times V(K_{n_t})$.
 For each $i=1,2,\ldots,t$,
 pick two distinct vertices $x_i,y_i\in V(K_{n_i})$.
 Let $x=(x_1,x_2,\ldots,x_t)$.
 For $1\leq j\leq t$, let $p^j=(p^j_1,\ldots ,p^j_t)$ be a vertex in $V$
 such that $p^j_i=x_i$ when $i\not= j$, and $p^j_j=y_j$.
 For $1\leq j\leq t-1$, let $q^j=(q^j_1,\ldots ,q^j_t)$ be a vertex in $V$
 such that $q^j_i=x_i$ when $i\not\in \{j,j+1\}$, and $q^j_j=y_j$,
 $q^j_{j+1}=y_{j+1}$.

 First, we want to show that $1+\lceil{t \over 2}\rceil$
 is an upper bound for {\rm min-seed}$(G,2)$. The proof is divided into
 two cases according to the parity of $t$.\\
 {\bf Case 1}. $t=2\ell$. Let
 $S=\{p^1,p^2\}\cup\{q^3,q^5,q^7,\ldots,q^{t-1}\}$.
 By Lemma \ref{key_observation}(d) it can be seen that
 $[\{p^1,p^2\}]_2^G=p^1_{[3,t]}=x_{[3,t]}$,
 $[\{p^1,p^2,q^3\}]_2^G=[x_{[3,t]}\cup q^3_{[1,t]}]_2^G=x_{[5,t]}$,
 and $[\{p^1,p^2,q^3,q^5\}]_2^G=[x_{[5,t]}\cup
 q^5_{[1,t]}]_2^G=x_{[7,t]}$.
 Continue in this way, we obtain $[S]_2^G=[x_{[t-1,t]}\cup
 q^{t-1}_{[1,t]}]_2^G=x_\emptyset =V$,
 which means that
 {\rm min-seed}$(G,2)\leq |S|=\ell+1=1+\lceil{t \over 2}\rceil$.\\
 {\bf Case 2}. $t=2\ell+1$. Let
 $S=\{p^1,p^2,p^3\}\cup \{q^4,q^6,q^8,\ldots,q^{t-1}\}$.
 By Lemma \ref{key_observation}(d) and the same arguments as above,
 we obtain
 $[S]_2^G=[x_{[t-1,t]}\cup
 q^{t-1}_{[1,t]}]_2^G=V$, and hence
 {\rm min-seed}$(G,2)\leq |S|=\ell+2=1+\lceil{t \over 2}\rceil$.

   To show that $1+\lceil{t \over 2}\rceil$ is also a lower bound
   bound for {\rm min-seed}$(G,2)$, let $S$ be an optimal target set
   for $(G,2)$.  Since $[S]_2^G=V=x_\emptyset$, by Theorem
   \ref{key-theorem-for-Hamming}(b), we have $|\emptyset|\geq
   (2+t)-2|S|$, that is $|S|\geq 1+{t\over 2}$. Which completes the
   proof of the theorem.
 \end{proof}

\end{document}